\documentclass[hyperref]{article}
\usepackage{graphicx} 
\usepackage{float} 
\usepackage{amssymb}
\usepackage{amsmath}
\usepackage{mathtools}
\usepackage{xcolor}
\usepackage{verbatim}
\usepackage[thmmarks, amsmath]{ntheorem} 
\usepackage{bm} 
\usepackage[colorlinks, linkcolor=cyan, citecolor=cyan]{hyperref}
\usepackage{extarrows}
\usepackage[hang, flushmargin]{footmisc} 
\usepackage{mathrsfs} 
\usepackage[font=footnotesize, skip=0pt, textfont=rm, labelfont=rm]{caption, subcaption} 
\usepackage{booktabs} 
\usepackage{tocloft}
{
\theoremstyle{nonumberplain}
\theoremheaderfont{\bfseries}
\theorembodyfont{\normalfont}
\theoremsymbol{\mbox{$\Box$}}
\newtheorem{proof}{Proof}
}
\usepackage{theorem}
\newtheorem{theorem}{Theorem}[section]
\newtheorem{lemma}{Lemma}[section]

{
\theoremheaderfont{\bfseries}
\theorembodyfont{\normalfont}

}
\numberwithin{equation}{section} 
\graphicspath{{figure/}} 
\usepackage[a4paper]{geometry} 
\begin{document}	
	\title{\bf Asymptotic Normality of the Largest Eigenvalue for Noncentral Sample Covariance Matrices}
	\author{
		Huihui Cheng, 
		Minjie Song		
		\thanks{E-mail addresses: chenghuihui@ncwu.edu.cn (H.H.Cheng), songminjie0808@163.com (M.J. Song).}
		\medskip
		\\
		\textsuperscript{}School of Mathematics and Statistics,\\ North China University of Water Resources and Electric Power, Zhengzhou, China. 	}
\maketitle  
{\noindent \small{\bf Abstract:} Let $X$ be a $p\times n$ independent identically distributed real Gaussian matrix with positive mean $\mu $ and variance $\sigma^2$ entries. The goal of this paper is to investigate the largest eigenvalue of the noncentral sample covariance matrix $ W=XX^{T}/n$, when the dimension $p$ and the sample size $n$ both grow to infinity with the limit $p/n=c\,(0<c< \infty)$. Utilizing the von Mises iteration method, we derive an approximation of the largest eigenvalue $\lambda_{1}(W)$ and show that $\lambda_{1}(W)$ asymptotically has a normal distribution with expectation $p\mu ^2+(1+c)\sigma^2$ and variance $4c\mu^2\sigma^2$. 
	
}  
\maketitle
\section{Introdution}

Suppose that $\left\{X_{i j},1\leq i\leq p,1\leq j\leq n \right\} $ are independent and identical distributed real random variables with mean $\mathbb{E}\, (X_{1 1})=\mu$ and a finite positive variance $ \mathrm{Var}\,(X_{1 1}) =\sigma ^2 $. Let $ W=(W_{i j})_{1\leq i,j\leq p}=XX^{T}/n$ be the sample covariance matrix, where $X=(X_{i j})_{1\leq i\leq p,1\leq j\leq n}$. 

The eigenvalues of a $p\times p$ real matrix $A $ are the roots in $\mathbb{C}$ of its characteristic polynomial. We label them as $\lambda_{1}(A),\cdots , \lambda_{p}(A)$ in such a way that $|\lambda_{1}(A)|\geq \cdots \geq |\lambda_{p}(A)|$ with increasing phases. The singular values of $A$ are $s_k(A)=\sqrt{\lambda_{k}({A^TA})} (1\leq k \leq p)$, where $A^{T}$ denotes the transpose. Since the matrix $W$ is symmetric positive semidefinite, we have $\lambda_{1}(W)\geq\\cdots\geq\lambda_{p}(W)\geq0$. For simplicity, we use the abbreviations a.s. for almost surely, i.i.d. for independent and identical distributed, and the notation $n>>1$ indicates that $n$ is sufficientlly large.

Let $\mu_W=\frac{1}{p}\sum_{j=1}^{p} \delta_{\lambda_j(W)}$ be the empirical spectral distribution of $W$. The Mar\v{c}henko-Pastur theorem (see Chapter 3 of \cite{bai10}) states that  
\begin{equation*}
	\mu_W\xrightarrow[n\rightarrow \infty]{\mathcal{C}_b} \mathcal{F}_{MP},\quad \text{a.s.}
\end{equation*}
where “$\xrightarrow{\mathcal{C}_b}$” denotes the weak convergence of probability measures with respect to bounded continuous functions, and $\mathcal{F}_{MP}$ is the Mar\v{c}henko-Pastur law (M-P law) with  density function
\begin{equation*}
	f_{MP}(x) = 
	\begin{cases} 
		\dfrac{1}{2\pi xc\sigma^2}\sqrt{
			(b-x)(x-a)}, &\text{if } a\leq x \leq b, \\
		0, &\text{otherwise},
	\end{cases}	
\end{equation*}
and has a point mass $1-1/c$ at $x=0$ if $c>1$, where $a=\sigma^2(1-\sqrt{c})^2$ and $b=\sigma^2(1+\sqrt{c})^2$. In the case that $\sigma^2=1$, the M-P law is called the standard M-P law. It has the $k$-th moment $m_{k}=(k+1)^{-1}C_{2k}^{k} (k=0,1,2,\dots)$, which is a special case of Fuss–Catalan numbers (see \cite{forre15}). Bai, Krishnaiah, Silverstein and Yin in \cite{bai88,yin88} proved  that the limit for $\lambda_{1}(W)$ exists if and only if the underlying distribution has zero mean and finite fourth moment, that is, 
\begin{equation}
	\lim\limits_{n\to\infty}\lambda_{1}\left( W \right)=\sigma^2(1+\sqrt{c})^2\;\;\text{a.s.} \quad\;\text{if and only if}\;\;\;\mu=0\;\; \text{and} \;\;\mathbb{E}(|X_{11}|^4)\textless \infty.\label{1.1}
\end{equation}
Johnstone \cite{john01} found the limiting distribution of the largest eigenvalue of the Wishart matrix. As stated in \cite{Mer19}, if $X_{1 1}$ is a standard normal random variable then we have $2^{-4/3}n^{2/3}\left(\lambda_{1}\left(W \right)-4\right)$ converges in distribution to Tracy-Widom distribution, which is introduced in \cite{Tracy96} to describe the fluctuations of the largest eigenvalues of the Gaussian Orthogonal Ensemble random matrices. And, Johnstone \cite{john01} initially proposed the spiked covariance model with relatively small spikes. Afterward some impressive studies are dedicated to investigating the limiting properties of the spiked eigenvalues with some assumptions, such as those by Baik, Ben Arous and Péché \cite{baik04}, Baik and Silverstein \cite{baik06}, Paul \cite{paul07}, Bai and Yao \cite{bai08}. Furthermore, Loubaton and Vallet \cite{Loubaton11} studied the almost sure location of the eigenvalues of matrix $D_{n}D_{n}^*$, where $D_n=B_n +\sigma X_n$, $B_n$ is a uniformly bounded
deterministic matrix such that $\sup_n\Vert B_n\Vert <\infty$ and $X_n$ is an independent identically distributed complex Gaussian matrix with zero mean and variance $1/n$ entries.

In this paper, we are interested in the asymptotic behavior of the largest eigenvalue $\lambda_{1}(W)$ for the general case of $\mu>0$. By denoting $\overline{X}=X-\mathbb{E}X$, we obtain $X/\sqrt{n}=\mathbb{E}X/\sqrt{n}+\overline{X}/\sqrt{n}$. This suggests that the largest singular value of $\mathbb{E}X/\sqrt{n}$ is $\sqrt{p}\mu$, which is distinct from the condition of uniform boundedness in \cite{Loubaton11}. Our primary result is as follows.
\begin{theorem}
	Let $\left\{X_{i j}, 1\leq i\leq p,1\leq j\leq n \right\}$ be i.i.d. real normal random variables with mean $\mu>0 $ and variance $\sigma^2$, and the noncentral sample covariance matrix $W=XX^{T}/n$. Then the distribution of $\lambda_{1}(W)$ asymptotically has a normal distribution with expectation $p\mu^2+(1+c)\sigma^2$ and variance $4c\mu^2\sigma^2$ of order $\sqrt{p}/n$. 
\end{theorem}

We observe that the largest eigenvalue $\lambda_{1}(\mathbb{E}W)$ is equal to $p\mu^2+\sigma^2$, while the other eigenvalues all equal to $\sigma^2$, and the vector  $\mathbf{1}=(1,1,\dots,1)^T$ is an eigenvector corresponding to $\lambda_{1}(\mathbb{E}W)$. Therefore, the vector $\mathbf{1}$ is “nearly” an eigenvector corresponding to $\lambda_{1}(W)$. Based on this, we use the von Mises iteration method \cite{furedi81} to approximate $\lambda_{1}(W)$ by approximating $\lambda_{1}^{(1)}= \sum_{i=1}^{p} W_i /p$, and  $\lambda_{1}^{(2)}= \dfrac{\sum_{i=1}^{p} W_i^{2}}{\sum_{i=1}^{p} W_i}$ in two steps, where  $W_i$ denotes the sum of the elements in the $i$-th row of $W$. The study in \cite{furedi81}  focused on the asymptotic behavior of the largest eigenvalue of a symmetric random matrix with bounded independent and identically distributed entries. Unlike the study of it, the entries of $W$ in our analysis are not independent and the mean depend on $p$. First, the error of $\lambda_{1}^{(2)}$ is of order $1/p$, that is, there exists a constant $C_1 $ such that when $n>>1$,
\begin{equation*}
	\mathbb{P} \left(\left|\lambda_{1}^{(2)}-\lambda_{1}(W)\right|>\dfrac{x}{p} \right)<\dfrac{C_1}{x^2}+O\left(\dfrac{1}{p}\right).
\end{equation*}
But, it is not easy to obtain the distribution of $\lambda_{1}^{(2)}$. Observe that  $\lambda_{1}^{(1)}$ asymptotically has a normal distribution. By using the relationship between $\lambda_{1}^{(2)}$ and  $\lambda_{1}^{(1)}$, we approximate $\lambda_{1}(W)$ with $\lambda_{1}^{(1)}+p\sigma^2/n$ of order $\sqrt{p}/n$, that is, when $n>>1$,
\begin{equation*}
		\mathbb{P} \left(\left\vert \lambda_{1}^{(1)}-\lambda_{1}(W)+\dfrac{p\sigma^2}{n}\right\vert>\dfrac{C_2 \sqrt{p}x}{n} \right)<\dfrac{C_3}{x^2}+O\left(\dfrac{1}{p}\right),
\end{equation*}
where  $C_2, C_3 $ are constants. 

In the course of the proof, we find that if $\mu>0$, when $n\to\infty$, $\lambda_{1}(W)$ is deviated from domain $[\sigma^2(1-\sqrt{c})^2,\sigma^2(1+\sqrt{c})^2]$ with high probability, while the second eigenvalue $\lambda_{2}(W)\leq \sigma^2(1+\sqrt{c})^2$, a.s.. 
Although the limit empirical spectral distribution of $W$ is independent of $\mu$, Theorem 1.1 shows that the distribution of $\lambda_{1}(W)$ in the case $\mu>0$ is differernt from the Tracy-Widom distribution in $\mu=0$.
\section{Preliminary results}

To prove the main theorem, we need some preliminary results.
\begin{lemma} Under the assumption of Theorem 1.1, $\lambda_{1}^{(1)}$ asymptotically has a normal distribution, and 
	\begin{equation}
		\mathbb{E}(\lambda_{1}^{(1)})=p\mu^2+\sigma ^2,\quad\mathrm{Var}(\lambda_{1}^{(1)})=4\mu^2\sigma^2\dfrac{p}{n}+O\left(\dfrac{1}{n}\right).\label{2.1}
	\end{equation}    
\end{lemma}
\begin{proof} Let $X_{\cdot j}=\sum_{i=1}^{p}X_{i j}, (1\leq j\leq n)$. Since $\left\{X_{i j}, 1\leq i\leq p,1\leq j\leq n \right\} $ are i.i.d. real normal random variables, then $\left\lbrace X_{\cdot j},1\leq j\leq n \right\rbrace$ are also i.i.d.. Reviewing that $\lambda_{1}^{(1)}=\dfrac{\sum_{i=1}^{p} W_i}{p}$, we have
	\begin{align*}
		\lambda_{1}^{(1)}=\dfrac{\sum_{j=1}^{n} X_{\cdot j}^2}{np}.
	\end{align*}
Based on the properties of expectation and variance, 
	\begin{align*}
		\mathbb{E}(\lambda_{1}^{(1)})=\dfrac{1}{p}\mathbb{E}( X_{\cdot 1}^2)=\dfrac{1}{p}\sum_{i=1}^{p}\mathbb{E}\left( X_{i 1}^2\right) +\dfrac{1}{p}\sum_{\substack{i,k=1 \\ i \neq k}}^{p} \mathbb{E}\left( X_{i 1}X_{k 1} \right) =p\mu ^2+\sigma^2,
	\end{align*}
and 
	\begin{align*}
		\mathrm{Var}(\lambda_{1}^{(1)})=\dfrac{1}{n^2p^2}\mathrm{Var}\left(\sum_{j=1}^{n} X_{\cdot j}^2\right)=\dfrac{1}{np^2}\mathrm{Var}(X_{\cdot 1}^2)=4\mu^2\sigma ^2\dfrac{p}{n}+O\left(\dfrac{1}{n}\right).
	\end{align*}
When $n>>1$, by calculating the characteristic function of $\lambda_{1}^{(1)}$, we prove that $\lambda_{1}^{(1)}$ asymptotically be a normal distribution with expectation $p\mu ^2+\sigma^2$ and variance $4c\mu^2\sigma^2$.
\end{proof}
\begin{lemma}Under the assumption of Theorem 1.1,
	\begin{equation}
		\mathbb{P} \left(\left|\sum_{i=1}^{n}\Big(W_i-(p\mu^2+\sigma^2)\Big)^2-\mathbb{E} \left( \sum_{i=1}^{n}\Big(W_i-(p\mu^2+\sigma^2)\Big)^2\right)\right|> \dfrac{C_4p^{5/2}x}{n} \right)<\dfrac{1}{x^2},\label{2.2}
	\end{equation}
and
	\begin{equation}
		\mathbb{P} \left(\left|\lambda_{1}^{(2)}-\lambda_{1}^{(1)}-\dfrac{p\sigma^2}{n} \right|>\dfrac{C_5 \sqrt{p}x}{n} \right)<\dfrac{1}{x^2}.\label{2.3}
	\end{equation} 
where $C_4, C_5$ are  constants.     
\end{lemma}
\begin{proof} Let $l:=p\mu^2+\sigma^2$ and $\overline{X_{i k }X_{j k }}=X_{i k }X_{j k }-\mathbb{E}(X_{i k }X_{j k }),(1\leq i,j\leq p,1\leq k\leq n)$. Then $\mathbb{E}(\overline{X_{i k }X_{j k }})=0$. 
Since
	\begin{align*}
		\sum_{i=1}^{p}(W_i-l)^2&=\dfrac{1}{n^2}\sum_{i=1}^{p}\left(
		\sum_{j=1}^{p}\left(\sum_{k=1}^{n}X_{i k }X_{j k }\right)-nl\right)^2=\dfrac{1}{n^2}\sum_{i=1}^{p}\left(\sum_{j=1}^{p}\sum_{k=1}^{n}\overline{X_{i k }X_{j k }}\right)^2,
	\end{align*}
we have
	\begin{align}
		\mathbb{E}\left(\sum_{i=1}^{p}(W_i-l)^2\right)&=\dfrac{1}{n^2}\sum_{k=1}^{n}\sum_{K=1}^{n}\sum_{i=1}^{p}\sum_{j=1}^{p}\sum_{J=1}^{p}\mathbb{E}\left( \overline{X_{i k }X_{j k }}\,\overline{X_{i K }X_{J K }}\right) \nonumber \\
		&=\dfrac{\mu^2\sigma^2p^3}{n}+O\left(\dfrac{p^2}{n}\right),\label{2.4}
	\end{align}
	\begin{align*}
		\mathbb{E}\left(\sum_{i=1}^{p}(W_i-l)^2\right)^2&=\dfrac{1}{n^4}\sum_{\substack{K,Q,\\k,q=1}}^{n}\,\sum_{\substack{I,J,G,\\i,j,g=1}}^{p}\mathbb{E}\left(\overline{X_{i k }X_{j k }}\,\overline{X_{i q }X_{g q }}\,\overline{X_{I K }X_{J K }}\,\overline{X_{I Q }X_{G Q }}\right)\\
		&=\dfrac{\mu^4\sigma^4p^6}{n^2}+O\left(\dfrac{p^5}{n^{2}}\right), 
	\end{align*}
and
	\begin{align*}
		\mathrm{Var}\left(\sum_{i=1}^{p}(W_i-l)^2\right)&=\mathbb{E}\left(\sum_{i=1}^{p}(W_i-l)^2\right)^2-\left(\mathbb{E}\sum_{i=1}^{p}(W_i-l)^2\right)^2=O\left(\dfrac{p^5}{n^{2}}\right).
	\end{align*}
Therefore, (\ref{2.2}) can be derived from Chebyshev's inequality. 
	
For (\ref{2.3}), we use the following indentity:
	\begin{align}
		\lambda_{1}^{(2)}-\lambda_{1}^{(1)}&=\dfrac{\sum_{i=1}^{p} W_i^{2}}{\sum_{i=1}^{p} W_i}- \dfrac{\sum_{i=1}^{p} W_i}{n}\nonumber\\
		&=\dfrac{\sum_{i=1}^{p}(W_i-l)^2}{\sum_{i=1}^{p} W_i}- \dfrac{\left(\left( \sum_{i=1}^{p}W_i/p\right) -l\right)^2}{\sum_{i=1}^{p}W_i/p}.\label{2.5}
	\end{align}
By using (\ref{2.2}) and (\ref{2.4}), it follows that 
	\begin{equation}
		\mathbb{P}\left(\left| \sum_{i=1}^{p}(W_i-l)^2- \dfrac{\mu^2\sigma^2p^3}{n}-O\left(\dfrac{p^2}{n}\right) \right|  <\dfrac{C_4p^{5/2}x}{n}  \right) \geq 1-\dfrac{1}{x^2}.\label{2.6}
	\end{equation}
Furthermore, by using (\ref{2.1}) and applying Chebyshev's inequality, we conclude that
	\begin{equation}
		\mathbb{P} \left(\left|\lambda_{1}^{(1)}- p\mu^2 -\sigma^2\right|<3\mu\sigma x\right)\geq 1-\dfrac{1}{x^2},\label{2.7}
	\end{equation}
and 
	\begin{equation}
		\mathbb{P} \left(\left|\sum_{i=1}^{p} W_i- p^2\mu^2 -p\sigma^2\right|<3p\mu\sigma x\right)\geq 1-\dfrac{1}{x^2}.\label{2.8}
	\end{equation}
In accordance with (\ref{2.6}) and (\ref{2.8}), we derive
	\begin{equation*}
		\mathbb{P} \left(\dfrac{p\sigma^2}{n}-\dfrac{C_4p^{5/2}x/n+3\mu\sigma^3x p^2/n +O\left(p^2/n\right)}{p^2\mu^2+p\sigma ^2+3\mu\sigma p x} < \dfrac{\sum_{i=1}^{p}(W_i-l)^2}{\sum_{i=1}^{p} W_i}\right)\geq 1-\dfrac{2}{x^2},
	\end{equation*}
and 
	\begin{equation*}
		\mathbb{P} \left(\dfrac{\sum_{i=1}^{p}(W_i-l)^2}{\sum_{i=1}^{p} W_i}< 
		\dfrac{p\sigma^2}{n}+\dfrac{C_4p^{5/2}x/n+3\mu\sigma^3x p^2/n +O\left(p^2/n\right)}{p^2\mu^2+p\sigma ^2+3\mu\sigma p x} \right)\geq 1-\dfrac{2}{x^2}.
	\end{equation*}
So, we have
	\begin{equation*}
		\mathbb{P} \left(\left|\dfrac{\sum_{i=1}^{p}(W_i-l)^2}{\sum_{i=1}^{p} W_i}-\dfrac{p\sigma^2}{n}\right|>\dfrac{C_6 \sqrt{p}x}{n} \right)\leq\dfrac{1}{x^2},
	\end{equation*}
where $C_6$ is a constant. In the same way, we notice that the second term on the right of (\ref{2.5}) is $O\left(\dfrac{1}{p}\right)$, that is, there exists constant $C_7$, such that
	\begin{equation*}
		\mathbb{P} \left(\left|\dfrac{\left(\left( \sum_{i=1}^{p}W_i/p\right) -l\right)^2}{\sum_{i=1}^{p}W_i/p}\right|>\dfrac{C_7 x^2}{p} \right)\leq\dfrac{1}{x^2}.
	\end{equation*}
Thus, we have
	\begin{equation*}
		\mathbb{P} \left(\left|\dfrac{\sum_{i=1}^{p}(W_i-l)^2}{\sum_{i=1}^{p} W_i}-\dfrac{\left(\left( \sum_{i=1}^{p}W_i/p\right) -l\right)^2}{\sum_{i=1}^{n}W_i/p}-\dfrac{p\sigma^2}{n}\right|>\dfrac{C_6 \sqrt{p}x}{n}+\dfrac{C_7 x^2}{p} \right)\leq\dfrac{2}{x^2}.
	\end{equation*}
Consequently, we obtain
	\begin{equation*}
		\mathbb{P} \left(\left|\lambda_{1}^{(2)}-\lambda_{1}^{(1)}-\dfrac{p\sigma^2}{n} \right|>\dfrac{C_5 \sqrt{p}x}{n} \right)<\dfrac{1}{x^2}.
	\end{equation*}
where $C_{5}$ is another constant.
\end{proof}
\section{ Proof of Theorem 1.1}

In this section, we derive an upper bound for
$\lambda_{2}(W)$ and prove the main theorem. Firstly, notice that if $\mu>0$, denoting $\overline{X}=(\overline{X}_{i j})_{p\times n}=X-\mathbb{E}X$,$\overline{W}=\dfrac{1}{n}\overline{X} \,\overline{X}^T$, we have
\begin{align*}
	W=\dfrac{1}{n}XX^T=\dfrac{1}{n}(\overline{X}+\mathbb{E}X)(\overline{X}+\mathbb{E}X)^T=\overline{W}+R,
\end{align*}
where 
\begin{equation*}
	R=\dfrac{1}{n}\left(\left( \mathbb{E}X\right) \overline{X}^T+\overline{X}(\mathbb{E}X)^T+\mathbb{E}X(\mathbb{E}X)^T\right).
\end{equation*}
Owing to $\overline{X}_{i j}=X_{i  j}-\mathbb{E}\left(X_{i j}\right) $, we have 
\begin{equation*}
	\mathbb{E} \left( \overline{X}_{i j}\right) =0,\;  \mathrm{Var} \left( \overline{X}_{i j}\right) =\sigma^2,\; \text{and}\; \mathbb{E}\left( \vert\overline{X}_{i j}\vert^4\right)<\infty.
\end{equation*}
From (\ref{1.1}), we  obtain that 
\begin{equation*}
	\lim\limits_{n \to\infty} \lambda_{1}\left( \overline{W} \right) =\sigma^2(1+\sqrt{c})^2. \;\;\; \text{a.s.}
\end{equation*}	
Since $\overline{W}$ and $R$ are symmetric, from the Monotonicity Theorem in \cite{ikebe87}, we have
\begin{align*}
	\lambda_1(\overline{W})\leq\lambda_1(W),
\end{align*}
and from singular inequality 
$s_{2}(X)\leq s_{1}(\overline{X})$,
we get
\begin{align}
	\lambda_2(W)\leq\lambda_1(\overline{W}).\label{3.1}
\end{align}
Thus, we obtain 
\begin{align*}	
	\lambda_2(W)\leq\lambda_1(\overline{W})\leq \lambda_1(W).
\end{align*}
This implies that the second largest eigenvalue $\lambda_2(W)$ is bounded by $\lambda_1(\overline{W})$, while the largest eigenvalue $\lambda_1(W)$ is not less than $\lambda_1(\overline{W})$.	Therefore, it is necessary to study the asymptotic behavior
of $\lambda_1(W)$.
Based on (\ref{3.1}), when $n\to\infty$, we get that the second eigenvalue $\lambda_{2}(W)\leq \sigma^2(1+\sqrt{c})^2$, a.s..
\begin{proof} 
For the sake of convenience, the largest eigenvalue $\lambda_{1}(W)$ be denoted as $\lambda_{1}$ in the following. Let $\mathbf{v}$ be the eigenvector corresponding to $\lambda_{1}$. We  apply the von Mises iteration method to approximate the largest eigenvalue $\lambda_1$. First, we split $\mathbf{1} $ into $\mathbf{v}$ and a component orthogonal to $\mathbf{v}$:
	\begin{equation}
		\mathbf{1}=\mathbf{v}+\mathbf{r},\,\, (\mathbf{v},\mathbf{r})=0,\,\, W\mathbf{v}=\lambda_{1}\mathbf{v}.\label{3.3}
	\end{equation}
So that
	\begin{equation}
		W\mathbf{1}=W\mathbf{v}+W\textbf{r}=\lambda_{1}\mathbf{v}+W\textbf{r}.\label{3.4}
	\end{equation}
Recalling $l=p\mu^2+\sigma^2$, subtract $l\mathbf{1}$ from both sides of (\ref{3.4}),
	\begin{equation}
		W\mathbf{1}-l\mathbf{1}=(\lambda_{1}-l)\mathbf{v}+(W\mathbf{r}-l\mathbf{r}).\label{3.5}
	\end{equation}
Since $\mathbf{r}$ is orthogonal to $\mathbf{v}$, we have
	\begin{equation}
		\sum_{i=1}^{p}(W_i-l)^2=\Vert W\mathbf{1}-l\mathbf{1}\Vert^2=(\lambda_{1}-l)^2\Vert \mathbf{v} \Vert ^2+\Vert W\mathbf{r}-l\mathbf{r}\Vert ^2.\label{3.6}
	\end{equation}
In the following, we use (\ref{2.2}) to obtain a good approximation for $\Vert\mathbf{r}\Vert $. Due to $(\mathbf{v},\mathbf{r})=0,$ and  $W\mathbf{v}=\lambda_{1}\mathbf{v}$, we have $W \mathbf{r}\leq\lambda_{2}(W)\mathbf{r}$. Thus, when $n>>1$, we have a.s.
	\begin{equation}
		\Vert W \mathbf{r}\Vert \leq \lambda_{2}(W)\Vert\mathbf{r}\Vert\leq  2 \sigma^2(1+\sqrt{c})^2\Vert\mathbf{r}\Vert,\label{3.7}
	\end{equation}
and
	\begin{equation*}
		\Vert W\mathbf{r}-l\mathbf{r}\Vert \geq (l- 2 \sigma^2(1+\sqrt{c})^2)\Vert\mathbf{r}\Vert.
	\end{equation*}
Setting $x=\sqrt{p}$ in (\ref{2.2}), we obtain 
	\begin{equation*}
		\Vert\mathbf{r}\Vert ^2 \leq \dfrac{\Vert W\mathbf{1}-l\mathbf{1}\Vert^2}{(l- 2 \sigma^2(1+\sqrt{c})^2)^2}<\dfrac{2C_8p^3/n}{(l- 2 \sigma^2(1+\sqrt{c})^2)^2}<\dfrac{4C_9}{\mu ^4}, \text{\quad with probability at least $ 1-\dfrac{1}{p}$},
	\end{equation*}
where $C_8, C_9$ are constants. That is, we have 
	\begin{equation}
		\mathbb{P}\left(\Vert\mathbf{r}\Vert ^2<\dfrac{4C_9}{\mu ^4} \right)\geq 1-\dfrac{1}{p}.\label{3.8}
	\end{equation}
Now, applying the von Mises iteration method we  estimate $\lambda_1$ with $\lambda_1^{(2)}$ and presenting the error
	\begin{align}
		\lambda_{1}^{(2)}=\dfrac{\Vert W\mathbf{1}\Vert ^2}{(\mathbf{1},W\mathbf{1})} =\dfrac{\Vert \lambda_{1}\mathbf{v}+W\textbf{r} \Vert ^2}{(\mathbf{v}+\mathbf{r},\lambda_{1}\mathbf{v}+W\textbf{r})}=\lambda_{1}+\dfrac{\Vert W\textbf{r}\Vert ^2-\lambda_{1}\mathbf{r}^T W\textbf{r}}{\sum_{i=1}^{p} W_i}.\label{3.9}
	\end{align}	
From (\ref{3.7}) and (\ref{3.8}), we derive 
	\begin{equation}
		\mathbb{P}\left(\Vert W \mathbf{r}\Vert^2 <\dfrac{100C_9\sigma ^4}{\mu^4}\right)\geq 1-\dfrac{1}{p},\label{3.10}
	\end{equation}
and 
	\begin{equation}
		\mathbb{P}\left(| \mathbf{r}^T W \mathbf{r} |  <\dfrac{20C_9\sigma ^2}{\mu^4}\right)\geq 1-\dfrac{2}{p}.\label{3.11}
	\end{equation}
that is, $\mathbf{r}^TW\mathbf{r}$ is bounded in probability. In accordance with (\ref{3.11}) and Chebyshev's inequality, we obtain that there exist constants $C_{10}$ and $C_{11}$, such that
	\begin{align}
		\mathbb{P}\left(| \mathbf{r}^T W \mathbf{r} |>C_{10}+x\right)&\leq \mathbb{P}\left(| \mathbb{E}\left( \mathbf{r}^T W \mathbf{r}\right) |\geq C_{10}\right)+\mathbb{P}\left(|\mathbf{r}^T W \mathbf{r}-\mathbb{E}(\mathbf{r}^T W \mathbf{r})|\geq x\right) \nonumber \\ 
		&\leq\dfrac{C_{11}}{x^2}+\dfrac{2}{p}.\label{3.12}
	\end{align}	
Indeed, $\lambda_{1}$ is always estimated by
	\begin{align*}
		\lambda_{1}\leq\max\limits_{i}\sum_{j=1}^{p}\left|W_{i j}\right| \leq\max\limits_{i}\sum_{j=1}^{p}\sum_{k=1}^{n}\dfrac{1}{n}\left|X_{i k }X_{j k }\right|.
	\end{align*}
Thus, it is easy to show that 
	\begin{equation}
		\mathbb{P}\left(\lambda_{1}<C_{12} p \right)\geq 1-\dfrac{1}{p},\label{3.13}
	\end{equation}
where $C_{12}$ is a constant. By combining (\ref{2.8}) and (\ref{3.10}), there exists a constant $C_{13}$, such that 
	\begin{equation*}
		\mathbb{P}\left(\dfrac{\Vert W \mathbf{r}\Vert^2}{\left\vert\sum_{i=1}^{p} W_i\right\vert} >\dfrac{C_{13} x}{p^2}\right)\leq\dfrac{1}{x^2}+\dfrac{1}{p}.
	\end{equation*}
Similarly, by using (\ref{2.8}) and (\ref{3.12}--\ref{3.13}), we induce that
	\begin{equation*}
		\mathbb{P}\left(\left\vert\dfrac{\lambda_{1}\mathbf{r}^T W\textbf{r}}{\sum_{i=1}^{p} W_i} \right\vert >\dfrac{C_{14} x}{p}\right)\leq\dfrac{C_{15}}{x^2}+\dfrac{3}{p}.
	\end{equation*}
Since 
	\begin{equation*}
		\vert\lambda_{1}^{(2)}-\lambda_{1}\vert=\left\vert \dfrac{\Vert W\textbf{r}\Vert ^2-\lambda_{1}\mathbf{r}^T W\textbf{r}}{\sum_{i=1}^{p} W_i} \right\vert \leq \dfrac{\Vert W \mathbf{r}\Vert^2}{\left\vert\sum_{i=1}^{p} W_i\right\vert} +\left\vert\dfrac{\lambda_{1}\mathbf{r}^T W\textbf{r}}{\sum_{i=1}^{p} W_i} \right\vert,
	\end{equation*}
we  obtain 
	\begin{align*}
		\mathbb{P} \left(|\lambda_{1}^{(2)}-\lambda_{1}|>\dfrac{C_{13} x}{p^2}+\dfrac{C_{14} x}{p} \right)&\leq	\mathbb{P}\left(\dfrac{\Vert W \mathbf{r}\Vert^2}{\left\vert\sum_{i=1}^{p} W_i\right\vert} >\dfrac{C_{13} x}{p^2}\right)+\mathbb{P}\left(\left\vert\dfrac{\lambda_{1}\mathbf{r}^T W\textbf{r}}{\sum_{i=1}^{p} W_i} \right\vert >\dfrac{C_{15} x}{p}\right) \\ &\leq \dfrac{1+C_{15}}{x^2}+\dfrac{3}{p}
		.
	\end{align*}
Therefore, we show that the remaining term on the right of (\ref{3.9}) is of order $O\left(\dfrac{1}{p}\right)$, that is, there exists a constant $C_1$, such that 
	\begin{equation}
		\mathbb{P} \left(|\lambda_{1}^{(2)}-\lambda_{1}|>\dfrac{x}{p} \right)<\dfrac{C_1}{x^2}+O\left(\dfrac{1}{p}\right).\label{3.14}
	\end{equation}
By virtue of 
	\begin{equation*}
		\vert \lambda_{1}^{(1)}-\lambda_{1}+\dfrac{p\sigma^2}{n}\vert \leq \vert \lambda_{1}^{(2)}-\lambda_{1}\vert+\vert\lambda_{1}^{(2)}-\lambda_{1}^{(1)}-\dfrac{p\sigma^2}{n}\vert,
	\end{equation*}
and using (\ref{2.3}), we deduce 
	\begin{align*}
		\mathbb{P} \left(\vert \lambda_{1}^{(1)}-\lambda_{1}+\dfrac{p\sigma^2}{n}\vert>\dfrac{x}{p}+\dfrac{C_5 \sqrt{p}x}{n} \right)&\leq		\mathbb{P} \left(|\lambda_{1}^{(2)}-\lambda_{1}|>\dfrac{x}{p} \right)+	\mathbb{P} \left(\left|\lambda_{1}^{(2)}-\lambda_{1}^{(1)}-\dfrac{p\sigma^2}{n} \right|>\dfrac{C_5 \sqrt{p}x}{n}\right) \\ &< \dfrac{C_1+1}{x^2}+O\left(\dfrac{1}{p}\right).
	\end{align*}
As a result, there exist constants $C_2$ and $C_3$ such that
	\begin{equation*}
		\mathbb{P} \left(\vert \lambda_{1}^{(1)}-\lambda_{1}+\dfrac{p\sigma^2}{n}\vert>\dfrac{C_2 \sqrt{p}x}{n} \right)<\dfrac{C_3}{x^2}+O\left(\dfrac{1}{p}\right).
	\end{equation*}
From Lemma 2.1, we conclude that when $n>>1$, $\lambda_{1}$ converges asymptotically to a normal distribution with expectation $p\mu ^2+(1+c)\sigma^2$ and variance $4c\mu^2\sigma^2$ of order $\sqrt{p}/n$.
\end{proof}

\section*{Acknowledgement}

We would like to express our sincere thanks to Dang-Zheng Liu for his discussion and suggestion. 

\bibliography{reference}
\bibliographystyle{plain}

\end{document}